\documentclass[11pt]{article}
\usepackage[a4paper,margin=28mm]{geometry}
\usepackage{amsmath,amssymb,amsthm,aliascnt,mathtools}
\usepackage{microtype}
\usepackage{xcolor}
\AtBeginDocument{\microtypesetup{activate=false}}
\usepackage{hyperref}
\hypersetup{colorlinks=true,linkcolor=cyan,citecolor=red,urlcolor=red,filecolor=red,
pdftitle={Real-rooted Ehrhart h-star-polynomials of crown posets and wheel and whirl matroids},pdfauthor={Ying Cao and Houshan Fu}}
\numberwithin{equation}{section}

\newtheorem{theorem}{Theorem}[section]
\newaliascnt{lemma}{theorem}
\newtheorem{lemma}[lemma]{Lemma}
\aliascntresetthe{lemma}
\newaliascnt{proposition}{theorem}
\newtheorem{proposition}[proposition]{Proposition}
\aliascntresetthe{proposition}
\newaliascnt{corollary}{theorem}
\newtheorem{corollary}[corollary]{Corollary}
\aliascntresetthe{corollary}
\theoremstyle{remark}
\newaliascnt{remark}{theorem}
\newtheorem{remark}[remark]{Remark}
\aliascntresetthe{remark}

\DeclareMathOperator{\conv}{conv}
\DeclareMathOperator{\tr}{tr}
\DeclareMathOperator{\ehr}{ehr}
\newcommand{\R}{\mathbb R}
\newcommand{\Z}{\mathbb Z}
\newcommand{\OO}{\mathcal O}
\newcommand{\CC}{\mathcal C}
\newcommand{\one}{\mathbf 1}
\allowdisplaybreaks[1]

\begin{document}
\begin{center}
{\Large\bfseries Real-rooted Ehrhart $h^*$-polynomials of crown posets\par
and wheel and whirl matroids\par}
\end{center}

\begin{center}
Ying Cao$^1$\quad\quad Houshan Fu$^{2,*}$\par\vspace{.2cm}

$^1$Department of Mathematics, Hong Kong University of Science and Technology\\
Clear Water Bay, Hong Kong\\\par
\vspace{.2cm}
$^2$School of Mathematics and Information Science, Guangzhou University\\
Guangzhou 510006, Guangdong, P. R. China\par
\vspace{.2cm}
$^*$Correspondence to be sent to: fuhoushan@gzhu.edu.cn \\
Email: ycaobf@connect.ust.hk
\end{center}

\begin{abstract}
For every integer $r\ge3$, the crown order polytope on $2r$ elements and the rank-$r$ whirl base polytope have the same Ehrhart $h^*$-polynomial. We also prove that every convex combination of the $h^*$-polynomials of the rank-$r$ wheel and whirl base polytopes has degree $2r-2$ and only simple negative roots. Its coefficients are positive, unimodal, and strictly ultra log-concave of order $2r-2$. These results settle the real-rootedness conjecture of Lundstr\"om and Saud Maia Leite for crown order polytopes.


\end{abstract}

\section{Introduction}\label{sec:introduction}
For a lattice polytope $P\subseteq\R^n$ of dimension $d$, its \emph{Ehrhart polynomial} $\ehr(P,q):=|qP\cap\Z^n|$ counts the lattice points in its nonnegative integer dilates \cite{Ehrhart}. The \emph{Ehrhart series} of $P$ has the form
\begin{equation}\label{eq:ehrhart-series}
 \sum_{q\ge0}\ehr(P,q)t^q=\frac{h_P^*(t)}{(1-t)^{d+1}}.
\end{equation}
The numerator $h_P^*(t)$, called the \emph{Ehrhart $h^*$-polynomial} of $P$, has nonnegative integer coefficients \cite[Theorem~2.1]{StanleyNonnegative}. We study its real-rootedness for crown order
polytopes and for wheel and whirl base polytopes.

Order polytopes relate lattice-point enumeration to order-preserving maps and linear extensions \cite[Theorem~4.1 and Section~5]{StanleyPosets}. For an integer $r\ge2$, the \emph{crown poset} $\CC_{2r}$ is the poset on $[2r]:=\{1,\ldots,2r\}$ with cover relations
\[
1\prec2\succ3\prec\cdots\succ(2r-1)\prec2r\succ1.
\]
Its \emph{order polytope} is
\begin{equation}\label{eq:crown-polytope}
 \OO(\CC_{2r}):=\{\boldsymbol{x}\in[0,1]^{2r}:x_{2i-1}\le x_{2i}\ge x_{2i+1}\text{ for }1\le i\le r\},
\end{equation}
where $x_{2r+1}=x_1$.

Recently, Lundstr\"om and Saud Maia Leite determined that $h^*_{\OO(\CC_{2r})}$ has degree $2r-2$ \cite[Corollary~5.13]{Crown}. Using Br\"and\'en's theorem on sign-graded posets \cite[Theorem~4.2]{BrandenSignGraded}, they also established its $\gamma$-nonnegativity \cite[Theorem~5.19]{Crown}, and further conjectured that $h^*_{\OO(\CC_{2r})}$ is real-rooted \cite[Conjecture~5.18]{Crown}.

Without the cover relation $1\prec2r$, the crown poset becomes the zig-zag poset on $[2r]$. Coons and Sullivant interpreted the $h^*$-polynomials of zig-zag order polytopes through a swap statistic on alternating permutations and proved symmetry and unimodality \cite{CoonsSullivant}. Later, Petersen and Zhuang used Br\"and\'en's theorem to establish $\gamma$-nonnegativity for the associated zig-zag Eulerian polynomials \cite[Theorem~2.5]{PetersenZhuang}. More recently, Alexandersson proved that zig-zag order polytopes have $h^*$-polynomials with only simple negative roots \cite[Theorem~7.1]{AlexanderssonEulerian}.

Our original goal is to settle the real-rootedness conjecture for crown order polytopes. We first relate crown order polytopes to matroid base polytopes. For a matroid $M$ on a finite ground set $E$, its \emph{base polytope} is
\[
 P(M):=\conv\{\one_B\in\R^E:B\text{ is a basis of }M\},
\]
where $\one_B$ is the indicator vector of $B$.

For an integer $r\ge3$, let $W_r$ be the \emph{wheel graph} obtained from a cycle on $r$ vertices by adjoining a new vertex, called the \emph{hub}, adjacent to every vertex of the cycle. The original cycle
is called the \emph{rim}. The cycle matroid $M_r:=M(W_r)$ has rank $r$, and its bases are the edge sets of the spanning trees of $W_r$. The rim edge set is a circuit-hyperplane of $M_r$. The rank-$r$ \emph{whirl matroid} $N_r$ on $E(W_r)$ is obtained by relaxing this circuit-hyperplane. Its bases are those of $M_r$ together with the rim edge set.

For each integer $q\ge0$, the bijection in \autoref{thm:whirl-difference} identifies the lattice points of $qP(N_r)$ with those of $q\OO(\CC_{2r})$ having minimum coordinate zero. Comparing the Ehrhart series of the two polytopes then gives the following theorem.
\begin{theorem}\label{thm:crown-whirl}
For every integer $r\ge3$, we have
\[
 h^*_{P(N_r)}(t)=h^*_{\OO(\CC_{2r})}(t).
\]
\end{theorem}

\autoref{thm:crown-whirl} reduces the conjectured real-rootedness of $h^*_{\OO(\CC_{2r})}$ to that of $h^*_{P(N_r)}$. For general matroid base polytopes, De Loera, Haws, and K\"oppe conjectured unimodality of the $h^*$-coefficients \cite[Conjecture~2(A)]{DLHK}. Subsequently, Ferroni proposed the stronger real-rootedness conjecture \cite[Conjecture~1.3]{FerroniMinimal}. However, Fu found counterexamples among cycle matroids of theta graphs and, by duality, among rank-two matroids \cite[Theorem~1.1(2)--(3)]{Fu}.

Ferroni's formulas for minimal matroids and circuit-hyperplane relaxation \cite[Corollaries~1.7 and~5.8]{FerroniMinimal} express the difference between the $h^*$-polynomials of the wheel and whirl base polytopes as
\begin{equation}\label{eq:wheel-whirl-difference}
 h^*_{P(N_r)}(t)-h^*_{P(M_r)}(t)
 =t\sum_{i=0}^{r-1}\binom{r-1}{i}^{\!2}t^i.
\end{equation}
By \autoref{thm:crown-whirl}, subtracting a fraction of this difference from the crown $h^*$-polynomial interpolates between $h^*_{P(N_r)}$ and $h^*_{P(M_r)}$.

\begin{theorem}\label{thm:wheel-whirl}
For every integer $r\ge3$ and every $\lambda\in[0,1]$, the polynomial
\[
 h_{r,\lambda}(t)
 :=(1-\lambda)h^*_{P(N_r)}(t)+\lambda h^*_{P(M_r)}(t)
\]
has degree $2r-2$, and all its roots are simple and negative. In particular, the $h^*$-polynomials of $\OO(\CC_{2r})$, $P(M_r)$, and $P(N_r)$ have $2r-2$ simple negative roots.
\end{theorem}
For $r=2$, Lundstr\"om and Saud Maia Leite \cite[Example~5.14]{Crown} computed
\begin{equation}\label{eq:crown-rank-two}
h^*_{\OO(\CC_4)}(t)=1+2t+t^2=(1+t)^2.
\end{equation}
This polynomial has a double root at $-1$. Together with \eqref{eq:crown-rank-two}, the crown case of \autoref{thm:wheel-whirl} settles the conjecture of Lundstr\"om and Saud Maia Leite.

A nonnegative sequence $(a_i)_{i=0}^m$ is \emph{log-concave} if $a_i^2\ge a_{i-1}a_{i+1}$ for $1\le i<m$. It is \emph{ultra log-concave of order $m$} if it has no internal zeros and $(a_i/\binom mi)_{i=0}^m$ is log-concave. For a positive sequence, \emph{strict ultra log-concavity} requires strict inequalities at every internal index.

By \autoref{thm:wheel-whirl} and $h_{r,\lambda}(0)=1$, all coefficients of $h_{r,\lambda}$ are positive. Newton's inequalities \cite[Theorem~2]{StanleyLogConcavity} imply that its coefficient sequence is strictly ultra log-concave, since the roots are simple. In particular, the coefficient sequence is log-concave and unimodal.
\begin{corollary}\label{cor:strict-ulc}
Let $r\ge3$ be an integer and $\lambda\in[0,1]$. Write
$h_{r,\lambda}(t)=\sum_{i=0}^{2r-2}h_i(\lambda)t^i$.
Then $h_i(\lambda)>0$ for $0\le i\le2r-2$, and
\[
 \Bigl(\frac{h_i(\lambda)}{\binom{2r-2}{i}}\Bigr)^2
 >\frac{h_{i-1}(\lambda)}{\binom{2r-2}{i-1}}
  \frac{h_{i+1}(\lambda)}{\binom{2r-2}{i+1}}
 ,\qquad 1\le i\le2r-3.
\]
In particular, the coefficient sequences of the $h^*$-polynomials of $\OO(\CC_{2r})$, $P(M_r)$, and $P(N_r)$ are positive, strictly ultra log-concave of order $2r-2$, and unimodal.
\end{corollary}

For $0<\lambda<1$, the coefficient of $t$ in $h_{r,\lambda}$ is an integer minus $\lambda$ by \eqref{eq:wheel-whirl-difference}. Thus, these intermediate polynomials are not $h^*$-polynomials of lattice polytopes.

To prove real-rootedness for $r\ge4$, we choose negative points where the crown $h^*$-polynomial alternates in sign and dominates the difference in \eqref{eq:wheel-whirl-difference} in modulus. The same points work for $h_{r,\lambda}$ and its reciprocal polynomial, uniformly in $\lambda\in[0,1]$. The intermediate value theorem then supplies enough distinct negative roots to exhaust the degree.

\autoref{sec:matroids} proves the crown-whirl identity using a lattice-point bijection. \autoref{sec:crowns} establishes the counting formulas and coefficient symmetry needed for \autoref{thm:wheel-whirl}. We prove this theorem in \autoref{sec:wheel-roots}. We conclude in \autoref{sec:closure} with applications to even-cycle stable-set polytopes, products, and dilations.

\section{The crown-whirl correspondence}\label{sec:matroids}
Whirl matroids are multi-path matroids, whose bases admit the path representation of Bonin and Gim\'enez \cite[Example~1 and Theorem~3.6]{BoninGimenez}. Li and Shao \cite[Lemma~2.2]{LiShao} described the base polytopes of multi-path matroids using partial sums of coordinates and a common shift.

Fix $r\ge3$. For the whirl matroid $N_r$, we construct an explicit affine map from the crown order polytope $\OO(\CC_{2r})$ onto $P(N_r)$. Normalizing preimages of lattice points gives the bijection in \autoref{thm:whirl-difference}, from which \autoref{thm:crown-whirl} follows. By Stanley's vertex description \cite[Corollary~1.3]{StanleyPosets}, the vertices of an order polytope are precisely its $0/1$ points. Hence,
\[
 \OO(\CC_{2r})=\conv\bigl(\OO(\CC_{2r})\cap\{0,1\}^{2r}\bigr).
\]
Since affine maps preserve convex combinations, it suffices to construct an affine map sending the vertices of $\OO(\CC_{2r})$ onto the indicator vectors of the bases of $N_r$.

To describe the map on vertices, write $v_1,\ldots,v_r$ for the rim vertices of $W_r$ and $o$ for its hub. Let $f_i=ov_i$ and $e_i=v_iv_{i+1}$ for $1\le i\le r$, where $v_{r+1}=v_1$. The rim edge set is $R:=\{e_1,\ldots,e_r\}$. The spanning trees of $W_r$ are precisely those obtained by deleting at least one rim edge and joining each resulting rim component to the hub by exactly one spoke, with isolated rim vertices counted as components.

For a crown vertex $\boldsymbol{x}$, start with the rim edge set $R$. For $1\le i\le r$, add the spoke $f_i$ at a rise $x_{2i-1}<x_{2i}$ and delete the rim edge $e_i$ at a drop $x_{2i}>x_{2i+1}$. No edge is added or deleted at either the all-zero or the all-one vertex. In both cases, the resulting edge set is $R$ corresponding to the additional basis of $N_r$.

The spoke indicator equals $x_{2i}-x_{2i-1}$, and the rim-edge indicator equals $1-(x_{2i}-x_{2i+1})$. Scaling to the $q$-dilate replaces the constant $1$ by $q$. For each integer $q\ge0$, define the affine map $\Phi_q:\R^{2r}\to\R^{E(W_r)}$, $\boldsymbol{x}\mapsto\boldsymbol{y}$, by
\begin{equation}\label{eq:affine-map}
 y_{f_i}=x_{2i}-x_{2i-1},\qquad y_{e_i}=q+x_{2i+1}-x_{2i},\qquad 1\le i\le r.
\end{equation}
\begin{lemma}\label{lem:correspondence}
For every integer $q\ge0$, the map $\Phi_q$ sends $q\OO(\CC_{2r})$ onto $qP(N_r)$.
\end{lemma}
\begin{proof}
When $q=0$, both dilates are $\{\boldsymbol{0}\}$ and $\Phi_0(\boldsymbol{0})=\boldsymbol{0}$.
For $q\ge1$, the identity $\Phi_q(q\boldsymbol{x})=q\Phi_1(\boldsymbol{x})$ reduces the assertion to $q=1$.
For a crown vertex $\boldsymbol{x}$, \eqref{eq:affine-map} identifies $\boldsymbol{y}=\Phi_1(\boldsymbol{x})$ with the indicator vector of the edge set constructed above. In particular, both constant vertices map to $\one_R$. For a nonconstant vertex, the cyclic coordinate sequence has at least one drop. By \eqref{eq:crown-polytope}, a drop can occur only from $x_{2i}$ to $x_{2i+1}$, so at least one rim edge is deleted.

Consider a component $Q$ of the remaining rim subgraph, traversed from $v_a$ to $v_b$ in the direction $v_i\to v_{i+1}$. The rim edges immediately before and after $Q$ are deleted and may coincide. The corresponding drops from $1$ to $0$ force $x_{2a-1}=0$ and $x_{2b}=1$. Moreover, on each rim edge of $Q$, we have $x_{2i}=x_{2i+1}$. Then the sum counting the added spokes incident with $Q$ telescopes:
\[
 \sum_{v_i\in V(Q)}y_{f_i}
 =\sum_{v_i\in V(Q)}(x_{2i}-x_{2i-1})
 =x_{2b}-x_{2a-1}=1.
\]
Thus, each such component is joined to the hub by exactly one spoke, and the resulting edge set is a spanning tree of $W_r$.

Conversely, fix a spanning tree $T$ of $W_r$. For each component of the rim subgraph of $T$, denote the rim endpoint of its unique spoke by $v_j$. Reading that path in the direction $v_i\to v_{i+1}$, assign
\[
 (x_{2i-1},x_{2i})=
 \begin{cases}
 (0,0),&\text{if $v_i$ precedes $v_j$ on the path},\\
 (0,1),&\text{if $v_i=v_j$},\\
 (1,1),&\text{if $v_i$ follows $v_j$ on the path}.
 \end{cases}
\]
The pairs satisfy $x_{2i-1}\le x_{2i}$, with strict inequality exactly at the rim endpoints of the spokes of $T$. Across a rim edge of $T$, $x_{2i}=x_{2i+1}$, whereas across a missing rim edge, $(x_{2i},x_{2i+1})=(1,0)$. Hence, $\boldsymbol{x}$ is a crown vertex with $\Phi_1(\boldsymbol{x})=\one_{E(T)}$ by \eqref{eq:affine-map}. Consequently, taking convex hulls gives $\Phi_1(\OO(\CC_{2r}))=P(N_r)$.
\end{proof}

The value $\Phi_q(\boldsymbol{x})$ determines every consecutive coordinate difference of $\boldsymbol{x}$. Hence, two preimages of the same point differ by a constant vector. Subtracting the minimum coordinate from all coordinates singles out the preimage with minimum zero. An analogous choice of the smallest admissible shift appears in the proof of \cite[Lemma~6.6]{LiShao}.

\begin{theorem}\label{thm:whirl-difference}
For every integer $q\ge0$, the map $\Phi_q$ restricts to a bijection
\begin{equation}\label{eq:normalized-bijection}
 \Phi_q:\{\boldsymbol{x}\in q\OO(\CC_{2r})\cap\Z^{2r}:
             \min_{1\le j\le2r}x_j=0\}
 \xrightarrow{\ \sim\ }qP(N_r)\cap\Z^{E(W_r)}.
\end{equation}
The Ehrhart polynomials satisfy the identity
\begin{equation}\label{eq:whirl-difference}
 \ehr(P(N_r),q)=\ehr(\OO(\CC_{2r}),q)-\ehr(\OO(\CC_{2r}),q-1).
\end{equation}
\end{theorem}

\begin{proof}
Fix an integer $q\ge0$. Since the affine map in \eqref{eq:affine-map} has integral coefficients, it sends lattice points to lattice points. To prove surjectivity, fix $\boldsymbol{y}\in qP(N_r)\cap\Z^{E(W_r)}$ and take a real preimage $\boldsymbol{x}\in q\OO(\CC_{2r})$ using \autoref{lem:correspondence}. The consecutive coordinate differences in this preimage are integers by \eqref{eq:affine-map}:
\[
 x_{2i}-x_{2i-1}=y_{f_i}\in\Z,\qquad
 x_{2i+1}-x_{2i}=y_{e_i}-q\in\Z,\qquad 1\le i\le r.
\]
Choose $k$ with $x_k=\min_jx_j$. Summing consecutive differences shows that $x_i-x_k\in\Z$ for every $i$. Subtract $x_k$ from every coordinate. The new coordinates remain in $[0,q]$ because $0\le x_i-x_k\le x_i\le q$. Moreover, the common shift preserves the crown inequalities and the image under $\Phi_q$. Thus, the shifted vector is a lattice preimage with minimum coordinate zero. Any two preimages of $\boldsymbol{y}$ differ by a constant vector. If both have minimum zero, they coincide.

For $q\ge1$, split the lattice points of $q\OO(\CC_{2r})$ according to whether their minimum coordinate is zero or positive. Those with minimum zero correspond to the lattice points of $qP(N_r)$ by \eqref{eq:normalized-bijection}. Subtracting one from each coordinate maps the others bijectively onto the lattice points of $(q-1)\OO(\CC_{2r})$. Therefore, the two sides of \eqref{eq:whirl-difference} agree at every positive integer $q$. Since both sides are polynomials in $q$, \eqref{eq:whirl-difference} follows as a polynomial identity.
\end{proof}
\begin{proof}[Proof of \autoref{thm:crown-whirl}]
Note that the point with odd coordinates $1/3$ and even coordinates $2/3$ satisfies the inequalities in \eqref{eq:crown-polytope} strictly. Hence, $\OO(\CC_{2r})$ has dimension $2r$. By \eqref{eq:affine-map}, the kernel of the linear part of $\Phi_1$ is the line of constant vectors. Since $\Phi_1(\OO(\CC_{2r}))=P(N_r)$ by \autoref{lem:correspondence}, the rank-nullity theorem implies $\dim P(N_r)=2r-1$. Since both Ehrhart series have constant term $1$, multiplying \eqref{eq:whirl-difference} by $t^q$ and summing over $q\ge1$ gives
\[
 \sum_{q\ge0}\ehr(P(N_r),q)t^q=(1-t)\sum_{q\ge0}\ehr(\OO(\CC_{2r}),q)t^q=\frac{h^*_{\OO(\CC_{2r})}(t)}{(1-t)^{2r}}.
\]
It follows from \eqref{eq:ehrhart-series} and $\dim P(N_r)=2r-1$ that $h^*_{P(N_r)}(t)=h^*_{\OO(\CC_{2r})}(t)$.
\end{proof}
Using the bijection in \autoref{thm:whirl-difference}, we count the lattice points added by relaxation to recover Ferroni's formula \eqref{eq:wheel-whirl-difference}.
\begin{remark}
For an integer $q\ge1$, the inequality $\sum_{i=1}^r y_{e_i}\le q(r-1)$ cuts out
$qP(M_r)$ from $qP(N_r)$, as stated in \cite[Proposition~5.3]{FerroniMinimal}. Since all coordinates sum to $qr$, the added lattice points
have spoke sum $s=\sum_{i=1}^r y_{f_i}=qr-\sum_{i=1}^r y_{e_i}<q$.
The normalized preimage $\boldsymbol{x}$ of an added lattice point under \eqref{eq:normalized-bijection} determines nonnegative integer $r$-tuples
\[
a_i=x_{2i}-x_{2i-1}=y_{f_i},\qquad b_i=x_{2i}-x_{2i+1}=q-y_{e_i},\qquad 1\leq i\leq r,
\] 
with $\sum_{i=1}^r a_i=\sum_{i=1}^r b_i=s$. Conversely, for any such pair of tuples $(a_i)$ and $(b_i)$, start with $x_1=0$ and reconstruct the coordinates using successive differences $a_1,-b_1,\dots,a_r,-b_r$. The total increment is zero, and the coordinate range is bounded by the total rise $s$. Hence, subtracting the minimum from every coordinate gives the unique normalized preimage in $[0,s]^{2r}\subseteq[0,q]^{2r}$.

Counting these pairs and applying the double hockey-stick identity \cite[Lemma~A.3]{FerroniMinimal}, we obtain
\[
 \ehr(P(N_r),q)-\ehr(P(M_r),q)
 =\sum_{s=0}^{q-1}\binom{s+r-1}{r-1}^{\!2}=\sum_{i=0}^{r-1}\binom{r-1}{i}^{\!2}\binom{q+2r-2-i}{2r-1}.
\]
For $q=1$, the spoke-star vertex satisfies the rim inequality strictly. The inequality remains strict at a sufficiently close relative interior point of $P(N_r)$. Hence, $\dim P(M_r)=\dim P(N_r)=2r-1$. The last binomial coefficient is the coefficient of $t^q$ in $t^{i+1}/(1-t)^{2r}$. Multiplying the counting identity by $t^q$, summing over $q\ge0$, and comparing numerators in \eqref{eq:ehrhart-series} gives \eqref{eq:wheel-whirl-difference}.
\end{remark}

\section{Crown lattice-point counts and coefficient symmetry}\label{sec:crowns}

The crown Ehrhart series determines the $h^*$-polynomials of the wheel and whirl base polytopes by \autoref{thm:crown-whirl} and \eqref{eq:wheel-whirl-difference}. Lundstr\"om and Saud Maia Leite expressed the crown lattice-point counts recursively in terms of zigzag order polynomials \cite[Theorem~4.2]{Crown}. Replacing each even coordinate $x_{2i}$ by $q-x_{2i}$ identifies the lattice points of $q\OO(\CC_{2r})$ with the vertex weightings of a $2r$-cycle counted by the trace formula of B\'ona, Ju, and Yoshida \cite[Section~2.2]{BonaJuYoshida}. Here we fix the even coordinates, leaving independent choices for the odd coordinates. The transfer-matrix method \cite{KramersWannier,Montroll} expresses the resulting cyclic sum as a matrix trace, which we evaluate as a finite cosecant sum.

The counting formula involves the following finite cosecant sums.
Their even polynomial extensions will also allow us to recover the
degree and symmetry of the crown $h^*$-polynomial.
For integers $r\ge1$ and $n\ge1$, define
\begin{equation}\label{eq:cosecant-sum}
 S_r(n):=\sum_{i=0}^{n-1}
          \csc^{2r}\!\Bigl(\frac{(i+1/2)\pi}{n}\Bigr).
\end{equation}

The following formula is equivalent to Fried's count of cyclic up-down words of length $2r$ over $\{1,\ldots,q+2\}$ \cite[Theorem~3.1]{FriedCyclic}. The correspondence adds $1$ to the odd crown coordinates and $2$ to the even coordinates. For $q\ge2$, da Fonseca recorded the eigenvalues and eigenvectors of $(\min\{i,j\})_{1\le i,j\le q+1}$ \cite[Theorem~2.1]{DaFonseca}. We give a direct proof for use in \autoref{wh:endpoints}.

\begin{proposition}\label{prop:spectral-count}
For integers $r\ge2$ and $q\ge0$,
\begin{equation}\label{eq:crown-count}
 \ehr(\OO(\CC_{2r}),q)=\frac{S_r(2q+3)-1}{2\cdot4^r}.
\end{equation}
Equivalently,
\begin{equation}\label{eq:trace-spectrum}
 \ehr(\OO(\CC_{2r}),q)=\frac1{4^r}\sum_{j=1}^{q+1}
       \csc^{2r}\!\Bigl(\frac{(2j-1)\pi}{4q+6}\Bigr).
\end{equation}
\end{proposition}

\begin{proof}
Fix the even coordinates in $\{0,\ldots,q\}$. To index these choices from $1$ to $q+1$, write $u_i=x_{2i}+1$ for $1\le i\le r$, with $u_{r+1}=u_1$. By \eqref{eq:crown-polytope}, the odd coordinates can be chosen independently, with $0\le x_{2i+1}<\min(u_i,u_{i+1})$. Thus, each choice of the even coordinates contributes a cyclic product of entries of $\mathbf{K}_{q+1}:=(\min\{i,j\})_{1\le i,j\le q+1}$. The trace of $\mathbf{K}_{q+1}^r$ sums these products:
\[
 \ehr(\OO(\CC_{2r}),q)
 =\sum_{u_1,\ldots,u_r\in\{1,\ldots,q+1\}}\prod_{i=1}^r \min\{u_i,u_{i+1}\}
 =\tr(\mathbf{K}_{q+1}^r).
\]

We evaluate the trace by directly finding the eigenvectors of $\mathbf{K}_{q+1}$. For $0<\theta<\pi$ and $1\le i\le q+1$, summing the product-to-sum identities over $k$ and cancelling adjacent terms, we obtain
\[
 2\sin(\theta/2)\sum_{k=i}^{q+1}\sin(k\theta)
  =\cos((i-\tfrac12)\theta)-\cos((q+\tfrac32)\theta),\quad
 2\sin(\theta/2)\sum_{k=1}^{i}\cos((k-\tfrac12)\theta)
  =\sin(i\theta).
\]
To make $\cos((q+\tfrac32)\theta)$ vanish, choose $\theta_j=(2j-1)\pi/(2q+3)$ for $1\le j\le q+1$.
For $1\le i,j\le q+1$, apply the first identity to the inner sum below
and the second to the resulting cosine sum:
\[
 \sum_{k=1}^{q+1}\min\{i,k\}\sin(k\theta_j)
  =\sum_{\ell=1}^{i}\sum_{k=\ell}^{q+1}\sin(k\theta_j)
  =\frac{\sin(i\theta_j)}{4\sin^2(\theta_j/2)}.
\]
In the first equality, each term $\sin(k\theta_j)$ appears exactly $\min\{i,k\}$ times. Thus, the nonzero vector with $i$th entry $\sin(i\theta_j)$ is an eigenvector of $\mathbf{K}_{q+1}$ with eigenvalue $[4\sin^2(\theta_j/2)]^{-1}$. The angles satisfy $0<\theta_1<\cdots<\theta_{q+1}<\pi$, and $\sin(\theta/2)$ is strictly increasing on $(0,\pi)$. Hence, the $q+1$ eigenvalues are distinct and form the full spectrum. Summing their $r$th powers evaluates $\tr(\mathbf{K}_{q+1}^r)$ and proves \eqref{eq:trace-spectrum}.

Set $n=2q+3$ in \eqref{eq:cosecant-sum}. The term $i=q+1$ has angle $\pi/2$ and equals $1$. The remaining angles pair as $\alpha$ and $\pi-\alpha$, with equal sine. Each pair contains exactly one angle from \eqref{eq:trace-spectrum}. Hence,
\[
 S_r(2q+3)-1
   =2\sum_{j=1}^{q+1}\csc^{2r}\!\Bigl(\frac{(2j-1)\pi}{4q+6}\Bigr)
   =2\cdot4^r\,\ehr(\OO(\CC_{2r}),q).\qedhere
\]
\end{proof}

The argument $2q+3$ in \eqref{eq:crown-count} changes sign under $q\mapsto-3-q$. An even polynomial extension of $S_r$ will turn this sign change into a reflection identity for the Ehrhart polynomial. The Laurent-coefficient formula for $\csc^{2r}x$ and the polynomial formula for $S_r$ below also follow from Dowker's classical identities \cite[equations~(7)--(8)]{DowkerCosecs}. Here we derive the coefficients by differentiation and comparison of Laurent principal parts. The partial-fraction expansion of $\csc^2x$ then determines the polynomial extension of $S_r$. For each integer $r\ge1$, denote the coefficients of the principal part of the Laurent expansion of $\csc^{2r}x$ at zero by $a_{r,1},\ldots,a_{r,r}$:
\[
 \csc^{2r}x=\sum_{j=1}^r a_{r,j}x^{-2j}+O(1),
 \qquad x\to0.
\]

\begin{lemma}\label{lem:laurent}
For every integer $r\ge1$, $a_{r,r}=1$ and $a_{r,j}>0$ for
$1\le j\le r$. Moreover, $S_r(n)$ extends to an even polynomial
of degree $2r$:
\begin{equation}\label{eq:odd-sum}
 S_r(x)=2\sum_{j=1}^r a_{r,j}\Bigl(\frac{2x}{\pi}\Bigr)^{2j}
          \sum_{i=0}^{\infty}(2i+1)^{-2j}.
\end{equation}
\end{lemma}

\begin{proof}
Differentiating $\csc^{2k}x$ twice and using $1+\cot^2x=\csc^2x$, we have
\[
 \Bigl(\frac{d^2}{dx^2}+4k^2\Bigr)\csc^{2k}x
 =2k(2k+1)\csc^{2k+2}x,\qquad k\ge1.
\]
Repeated application of this identity for $k=1,\ldots,r-1$ gives
\[
 \Biggl[\prod_{k=1}^{r-1}\Bigl(\frac{d^2}{dx^2}+4k^2\Bigr)\Biggr]\csc^2x
 =\Biggl[\prod_{k=1}^{r-1}2k(2k+1)\Biggr]\csc^{2r}x
 =(2r-1)!\csc^{2r}x.
\]
Dividing by $(2r-1)!$ and expanding the operator product, we obtain
\begin{equation}\label{eq:cosecant-derivatives}
 \csc^{2r}x=\frac{1}{(2r-1)!}
 \sum_{j=1}^{r}4^{r-j}e_{r,r-j}\frac{d^{2j-2}}{dx^{2j-2}}\csc^2x,
\end{equation}
where $e_{r,k}$ is the $k$th elementary symmetric function of $1^2,\ldots,(r-1)^2$ for $0\le k<r$, with $e_{r,0}=1$. The difference $\csc^2x-x^{-2}$ is analytic at zero. Hence, comparing Laurent principal parts in this identity gives
\begin{equation}\label{eq:cosecant-coefficients}
 a_{r,j}=\frac{4^{r-j}e_{r,r-j}(2j-1)!}{(2r-1)!},\qquad 1\le j\le r.
\end{equation}
Consequently, $a_{r,j}>0$ for $1\le j\le r$ and $a_{r,r}=1$.

To evaluate $S_r(n)$, we use the classical partial-fraction expansion \cite[p.~118, equation~4.22.4]{NISTHandbook}
\[
 \csc^2x=\sum_{\ell\in\Z}(x-\ell\pi)^{-2},\qquad x\in\R\setminus\pi\Z.
\]
For integer $m\ge 0$, the $m$th derivatives of the summands are $O(|\ell|^{-m-2})$ as $|\ell|\to\infty$, uniformly on compact subsets of $\R\setminus\pi\Z$. Thus, the derivative series converge absolutely
and locally uniformly, so we may differentiate term by term. Substituting these series into \eqref{eq:cosecant-derivatives} and using \eqref{eq:cosecant-coefficients}, we obtain
\[
 \csc^{2r}x=\sum_{j=1}^r a_{r,j}\sum_{\ell\in\Z}(x-\ell\pi)^{-2j}.
\]
For an integer $n\ge1$, substitute $x=(k+1/2)\pi/n$ and sum over $0\le k<n$. Every integer has a unique representation $k-\ell n$ with $0\le k<n$ and $\ell\in\Z$. Absolute convergence permits this reindexing and the subsequent pairing of opposite half-integers:
\[
 S_r(n)=\sum_{j=1}^{r}a_{r,j}\Bigl(\frac n\pi\Bigr)^{2j}
 \sum_{m\in\Z}(m+1/2)^{-2j}
 =2\sum_{j=1}^{r}a_{r,j}\Bigl(\frac{2n}\pi\Bigr)^{2j}
 \sum_{i=0}^{\infty}(2i+1)^{-2j}.
\]
This proves \eqref{eq:odd-sum} for positive integers $n$. Its right-hand side is an even polynomial in $n$ of degree $2r$, since $a_{r,r}=1$.
\end{proof}

With $S_r$ extended by \eqref{eq:odd-sum}, the two sides of \eqref{eq:crown-count} are polynomials in $q$ that agree at every nonnegative integer. Hence, \eqref{eq:crown-count} holds as a polynomial identity in $q$. By \eqref{eq:cosecant-sum} and evenness, $S_r(1)=S_r(-1)=1$.
Evaluating \eqref{eq:crown-count} at $q=-1,-2$ gives the two zeros below.
For the reflection identity, replace $q$ by $-3-q$ and use the evenness of $S_r$:
\begin{equation}\label{eq:count-reflection}
 \ehr(\OO(\CC_{2r}),-1)=\ehr(\OO(\CC_{2r}),-2)=0,\qquad \ehr(\OO(\CC_{2r}),-3-q)=\ehr(\OO(\CC_{2r}),q).
\end{equation}
We use \eqref{eq:count-reflection} to recover the degree and coefficient symmetry established by Lundstr\"om and Saud Maia Leite \cite[Corollaries~5.13 and~5.16]{Crown}. The self-reciprocity in \autoref{lem:reciprocity} allows us to use the same comparison points for $h_{r,\lambda}$ and its reciprocal polynomial in \autoref{wh:signs}.

\begin{lemma}\label{lem:reciprocity}
For every integer $r\ge2$, $h^*_{\OO(\CC_{2r})}$ is monic of degree $2r-2$ and satisfies
\begin{equation}\label{eq:crown-reciprocity}
 t^{2r-2}h^*_{\OO(\CC_{2r})}(1/t)=h^*_{\OO(\CC_{2r})}(t).
\end{equation}
\end{lemma}

\begin{proof}
Write $G_r(t)=\sum_{q\ge0}\ehr(\OO(\CC_{2r}),q)t^q$, viewed as a rational function.
For $0\le k\le2r$, the generating function of $\binom qk$ is $t^k/(1-t)^{k+1}$. Evaluating it at $1/t$ and using $\binom{-q}{k}=(-1)^k\binom{q+k-1}{k}$, we obtain
\[
 \frac{t^{-k}}{(1-t^{-1})^{k+1}}=\frac{(-1)^{k+1}t}{(1-t)^{k+1}}=-\sum_{q\ge1}\binom{-q}{k}t^q.
\]
Expand the Ehrhart polynomial in the binomial basis $\binom qk$ for $0\le k\le2r$. Apply the displayed identity to each basis term.  The terms with $q=1,2$ vanish by \eqref{eq:count-reflection}. For $q\ge3$, reflection changes the Ehrhart argument from $-q$ to $q-3$:
\[
 G_r(1/t)
 =-\sum_{q\ge1}\ehr(\OO(\CC_{2r}),-q)t^q
 =-\sum_{q\ge3}\ehr(\OO(\CC_{2r}),q-3)t^q
 =-t^3G_r(t).
\]
Recall that $\dim\OO(\CC_{2r})=2r$. Substituting \eqref{eq:ehrhart-series} into $G_r(1/t)=-t^3G_r(t)$ and using $1-t^{-1}=-(1-t)/t$, we arrive at \eqref{eq:crown-reciprocity}. Finally, $h^*_{\OO(\CC_{2r})}(0)=1$, and the leading term of $h^*_{\OO(\CC_{2r})}$ is $t^{2r-2}$ via  \eqref{eq:crown-reciprocity}.
\end{proof}

\section{Real-rootedness along the wheel-whirl segment}\label{sec:wheel-roots}
We prove \autoref{thm:wheel-whirl} by sign comparisons that hold uniformly for $\lambda\in[0,1]$. Fix an integer $r\ge3$.
By \autoref{thm:crown-whirl}, $h_{r,0}(t)=h^*_{\OO(\CC_{2r})}(t)$.
For each integer $k\ge0$, define $B_k(t):=\sum_{i=0}^k\binom ki^2t^i$.
Substituting \eqref{eq:wheel-whirl-difference} into the definition of $h_{r,\lambda}$ in \autoref{thm:wheel-whirl}, we obtain
\begin{equation}\label{wh:interpolation-correction}
 h_{r,\lambda}(t)=h_{r,0}(t)-\lambda tB_{r-1}(t),\qquad 0\le\lambda\le1.
\end{equation}
\subsection{Sign comparison}\label{wh:signs}

Since $t\mapsto1/t$ interchanges $(-1,0)$ and $(-\infty,-1)$, it suffices to find $r-1$ distinct roots in $(-1,0)$ for each of $h_{r,\lambda}$ and its reciprocal polynomial. The same comparison points will serve for both.
The polynomial $tB_{r-1}(t)$ has degree $r<2r-2$ and vanishes at zero.
Hence, \eqref{wh:interpolation-correction} and
\autoref{lem:reciprocity} show that every $h_{r,\lambda}$ is monic
of degree $2r-2$ with constant term $1$.
Binomial symmetry implies $t^{r-1}B_{r-1}(1/t)=B_{r-1}(t)$. Using this identity and \eqref{eq:crown-reciprocity}, we replace $t$ by $1/t$ in \eqref{wh:interpolation-correction} and multiply by $t^{2r-2}$ to form the reciprocal polynomial:
\begin{equation}\label{wh:reversal}
 t^{2r-2} h_{r,\lambda}(1/t)=h_{r,0}(t)-\lambda t^{r-2}B_{r-1}(t).
\end{equation}
On $[-1,0]$, both $\lambda tB_{r-1}(t)$ and $\lambda t^{r-2}B_{r-1}(t)$ have modulus at most $|tB_{r-1}(t)|$, since $|t|^{r-2}\le|t|$ and $0\le\lambda\le1$. Thus, if $|h_{r,0}(t)|>|tB_{r-1}(t)|$, $h_{r,\lambda}$ and its reciprocal polynomial have the same sign as $h_{r,0}(t)$.

\begin{lemma}\label{lem:sign-comparison}
If there are points $0>t_1>\cdots>t_{r-1}=-1$ such that
\begin{equation}\label{wh:sign-comparison}
 (-1)^jh_{r,0}(t_j)>|t_jB_{r-1}(t_j)|,\qquad 1\le j\le r-1,
\end{equation}
then $h_{r,\lambda}$ has $2r-2$ simple negative roots for every
$\lambda\in[0,1]$.
\end{lemma}

\begin{proof}
Fix $\lambda\in[0,1]$. By \eqref{wh:sign-comparison}, the terms subtracted in \eqref{wh:interpolation-correction} and \eqref{wh:reversal} each have modulus less than $(-1)^j h_{r,0}(t_j)$. Therefore,
\[
 (-1)^jh_{r,\lambda}(t_j)>0,\qquad
 (-1)^jt_j^{2r-2}h_{r,\lambda}(1/t_j)>0,
 \qquad 1\le j\le r-1.
\]
Set $t_0=0$. The polynomials $h_{r,\lambda}(t)$ and $t^{2r-2}h_{r,\lambda}(1/t)$ both take the value $1$ at $t_0$, since $h_{r,\lambda}$ is monic with constant term $1$. Hence, the signs at consecutive points are opposite for both polynomials. By the intermediate value theorem, each polynomial has a root in every interval $(t_j,t_{j-1})$, for $1\le j\le r-1$. Taking reciprocals of the roots of the reciprocal polynomial supplies another $r-1$ distinct roots of $h_{r,\lambda}$ in $(-\infty,-1)$. Thus, $h_{r,\lambda}$ has $2r-2$ distinct negative roots. Since $\deg h_{r,\lambda}=2r-2$, all these roots are simple.
\end{proof}

It remains to find points satisfying \eqref{wh:sign-comparison}. For $r=3$, the term $tB_2(t)$ vanishes at the root $t_1=-2+\sqrt{3}$ of $B_2(t)=1+4t+t^2$ in $(-1,0)$. From \cite[Table~2]{Crown}, the crown $h^*$-polynomial $h_{3,0}$ is
\[
 h_{3,0}(t)=(1+7t+t^2)B_2(t)-6t^2.
\]
At $t_1$ and $t_2=-1$, this formula verifies the two required comparisons:
\[
 -h_{3,0}(t_1)=6t_1^2>0=|t_1B_2(t_1)|,\qquad h_{3,0}(-1)=4>2=|B_2(-1)|.
\]
By \autoref{lem:sign-comparison}, $h_{3,\lambda}$ has four simple negative roots for every $\lambda\in[0,1]$.

For $r\ge4$ and $1\le j\le r-1$, we choose the comparison point $t_j$ using the dominant-eigenvalue argument in \autoref{wh:endpoints} when $r\ge(j+1)(j+2)$. When $r<(j+1)(j+2)$, we use the nearest conjugate poles in \autoref{wh:central-range}. We verify the ordering of the combined sequence in \autoref{wh:completion}.

Both estimates use the following bound on $B_k$, for integers $k\ge0$:
\begin{equation}\label{wh:binomial-bound}
 |B_k(-x)|\le(1+x)^k,\qquad x\ge0.
\end{equation}
Indeed, the constant term of $(1+\mathrm{i}\sqrt{x}z)^k(1+\mathrm{i}\sqrt{x}z^{-1})^k$ is $B_k(-x)$.
Averaging over $z=e^{\mathrm{i}\theta}$ extracts this constant term:
\[
 B_k(-x)=\frac1{2\pi}\int_0^{2\pi}
       (1-x+2\mathrm{i}\sqrt{x}\cos\theta)^k\,d\theta.
\]
The bound \eqref{wh:binomial-bound} follows from the triangle inequality and
\[
 |1-x+2\mathrm{i}\sqrt{x}\cos\theta|^2
 =(1-x)^2+4x\cos^2\theta\le(1+x)^2.
\]

\subsection{The dominant-eigenvalue estimate}\label{wh:endpoints}

To make $h_{r,0}(-x)$ have sign $(-1)^j$ for small $x>0$, we choose $x$ so that the $q=j$ term determines the sign of the alternating Ehrhart series. We use the largest eigenvalues of the matrices $\mathbf{K}_{q+1}$ to estimate the tails with $q<j$ and $q>j$. We then compare the resulting lower bound for $(-1)^jh_{r,0}(-x)$ with $|xB_{r-1}(-x)|$.
For $0<x<1$, the Ehrhart series \eqref{eq:ehrhart-series} takes the absolutely convergent form
\begin{equation}\label{wh:alternating-ehrhart}
 \frac{h_{r,0}(-x)}{(1+x)^{2r+1}}
 =\sum_{q\ge0}(-1)^q\ehr(\OO(\CC_{2r}),q)x^q.
\end{equation}
For integer $q\ge0$, set
\[
 \Lambda_q:=\Bigl[4\sin^2\Bigl(\frac{\pi}{4q+6}\Bigr)\Bigr]^{-1},
\]
which is the largest eigenvalue of $\mathbf{K}_{q+1}$ in the proof of \autoref{prop:spectral-count}.

We first compare the Ehrhart count with the contribution $\Lambda_q^r$ of the largest eigenvalue.

\begin{lemma}\label{wh:trace-bound}
For integers $r\ge3$ and $q\ge0$, we have
\[
 \Lambda_q^r\le\ehr(\OO(\CC_{2r}),q)<(1+4^{-r})\Lambda_q^r.
\]
\end{lemma}

\begin{proof}
The diagonal entries of $\mathbf{K}_{q+1}$ sum to $(q+1)(q+2)/2$, so
\[
 \frac{\tr(\mathbf{K}_{q+1})}{\Lambda_q}
 =2(q+1)(q+2)\sin^2\Bigl(\frac{\pi}{4q+6}\Bigr)
 <\frac{\pi^2}{8}<\frac54.
\]
The first strict inequality follows from $\sin(\pi/(4q+6))<\pi/(4q+6)$ and $(q+1)(q+2)<(q+3/2)^2$. Write $\mu_1=\Lambda_q,\mu_2,\ldots,\mu_{q+1}$ for the positive eigenvalues of $\mathbf{K}_{q+1}$. By \eqref{eq:trace-spectrum} and the preceding bound,
\[
 0\le\frac{\ehr(\OO(\CC_{2r}),q)}{\Lambda_q^r}-1
 =\sum_{i=2}^{q+1}\Bigl(\frac{\mu_i}{\Lambda_q}\Bigr)^r
 \le\Bigl(\frac{\tr(\mathbf{K}_{q+1})}{\Lambda_q}-1\Bigr)^r
 <4^{-r}.
\]
The sum is empty when $q=0$. Multiplying by $\Lambda_q^r$ proves the assertion.
\end{proof}

The small-angle approximation for sine suggests replacing $\Lambda_q^r x^q$ by a constant multiple of $(q+3/2)^{2r}x^q$. The successive ratios of this sequence, $x(1+1/(q+3/2))^{2r}$, decrease strictly with $q$. We choose $x$ so that its terms at $q=j-1$ and $q=j+1$ are equal. The two ratios adjacent to $q=j$ are then reciprocals. Hence, $(q+3/2)^{2r}x^q$ is uniquely maximized at $q=j$ among nonnegative integers $q$. For integers $j\ge1$ and $r\ge(j+1)(j+2)$, this choice determines
\begin{equation}\label{wh:endpoint-node}
 x_j=\Bigl(\frac{2j+1}{2j+5}\Bigr)^r.
\end{equation}
We now bound the ratios of the eigenvalue weights $\Lambda_q^r x_j^q$ on either side of $j$.

\begin{lemma}\label{wh:weight-decay}
For integers $j\ge1$ and $r\ge(j+1)(j+2)$, take $x=x_j$ from \eqref{wh:endpoint-node}. Then
\[
 \frac{\Lambda_{q-1}^r x^{q-1}}{\Lambda_q^r x^q}<\frac{14}{25},\qquad  1\le q\le j,\qquad
 \frac{\Lambda_{q+1}^r x^{q+1}}{\Lambda_q^r x^q}<\frac25,\qquad q\ge j.
\]
\end{lemma}

\begin{proof}
Set $a=\pi/(4q+6)$ and $b=\pi/(4q+2)$, so that $0<a<b\le\pi/6$. By the alternating Taylor series for sine,
\[
 \frac{\sin a}{\sin b}
 <\frac{a-a^3/6+a^5/120}{b-b^3/6}
 =\frac ab\Bigl(1+\frac{b^2-a^2+a^4/20}{6-b^2}\Bigr).
\]
Since $q\ge1$ and $\pi^2<10$, we have $b/a\le5/3$, $a^2<1/10$, and $b^2<5/18$. Thus,
\[
 \frac{b^2-a^2+a^4/20}{6-b^2}
 <\frac{(16/9+1/200)a^2}{103/18}
 =\frac{3209\pi^2}{41200(2q+3)^2}<\frac4{5(2q+3)^2}.
\]
Consequently,
\[
 \frac{\sin(\pi/(4q+6))}{\sin(\pi/(4q+2))}
 <G(q):=\frac{2q+1}{2q+3}\Bigl(1+\frac4{5(2q+3)^2}\Bigr).
\]
The rational function $G$ given by this formula is strictly increasing on $[1,\infty)$, since $G'(u)=4(20u^2+52u+45)/(5(2u+3)^4)>0$ for real $u\ge1$. Hence, $G(q)\le G(j)$ for $1\le q\le j$. With $t=4/(2j+3)^2$, this bound and \eqref{wh:endpoint-node} imply
\[
 x^{-1/r}\frac{\Lambda_{q-1}}{\Lambda_q}
 <\frac{2j+5}{2j+1}G(j)^2
 =(1-t)(1+t/5)^2,\qquad 1\le q\le j.
\]
Similarly, since $\sin v/v$ decreases strictly on $(0,\pi)$ and $\pi/(4q+10)<\pi/(4q+6)$, we have
\[
 x^{1/r}\frac{\Lambda_{q+1}}{\Lambda_q}
 <\frac{2j+1}{2j+5}\Bigl(\frac{2q+5}{2q+3}\Bigr)^2
 \le1-t,\qquad q\ge j.
\]
It remains to bound the $r$th powers of these ratios. Since $r\ge t^{-1}-1/4$ and $0<t\le4/25$, we have $rt\ge1-t/4$. Combining this bound with $\log(1-t)<-t-t^2/2$ and $\log(1+t/5)<t/5$, we derive
\[
\begin{aligned}
 r\log\bigl[(1-t)(1+t/5)^2\bigr]
 &<-(1-t/4)(3/5+t/2)<-3/5<\log(14/25),\\
 r\log(1-t)
 &<-(1-t/4)(1+t/2)<-1<\log(2/5).
\end{aligned}
\]
Exponentiating and applying the preceding ratio bounds proves the lemma.
\end{proof}

To prove \eqref{wh:sign-comparison} in the range $r\ge(j+1)(j+2)$, we combine \autoref{wh:trace-bound} with \autoref{wh:weight-decay} to control the two tails of \eqref{wh:alternating-ehrhart}. We then estimate $|xB_{r-1}(-x)|$ on the same scale.

\begin{lemma}\label{wh:interior-endpoint}
For integers $j\ge1$ and $r\ge(j+1)(j+2)$, the point $x=x_j$ from \eqref{wh:endpoint-node} satisfies
\[
 (-1)^jh_{r,0}(-x)>|xB_{r-1}(-x)|.
\]
\end{lemma}

\begin{proof}
By \autoref{wh:trace-bound}, replacing $\ehr(\OO(\CC_{2r}),q)$ by $\Lambda_q^r$ in \eqref{wh:alternating-ehrhart} introduces an error of modulus less than $4^{-r}\sum_{q\ge0}\Lambda_q^r x^q$.
In $(-1)^j\sum_{q\ge0}(-1)^q\Lambda_q^r x^q$, \autoref{wh:weight-decay} shows that term magnitudes decrease as $q$ moves away from $j$ in either direction. The alternating-series estimate therefore bounds the tails with $q<j$ and $q>j$ below by $-\Lambda_{j-1}^r x^{j-1}$ and $-\Lambda_{j+1}^r x^{j+1}$, respectively.
Applying the tail estimates and the replacement-error bound to \eqref{wh:alternating-ehrhart}, we obtain
\begin{equation}\label{wh:alternating-tail-bound}
 \frac{(-1)^jh_{r,0}(-x)}{(1+x)^{2r+1}}
 >\Lambda_j^r x^j-\Lambda_{j-1}^r x^{j-1}-\Lambda_{j+1}^r x^{j+1}
 -4^{-r}\sum_{q\ge0}\Lambda_q^r x^q.
\end{equation}
In the positive sum $\sum_{q\ge0}\Lambda_q^r x^q$, \autoref{wh:weight-decay} bounds the tails with $q<j$ and $q>j$ by geometric series with ratios $14/25$ and $2/5$, respectively. Hence,
\[
 \frac{\sum_{q\ge0}\Lambda_q^r x^q}{\Lambda_j^r x^j}
 <1+\frac{14}{11}+\frac23<3.
\]
Dividing \eqref{wh:alternating-tail-bound} by $\Lambda_j^r x^j$ and using
\autoref{wh:weight-decay} together with the preceding sum bound, we obtain
\[
 \frac{(-1)^jh_{r,0}(-x)}{(1+x)^{2r+1}\Lambda_j^r x^j}
 >1-\frac{14}{25}-\frac25-3\cdot4^{-r}
 =\frac1{25}-3\cdot4^{-r}>\frac1{50}.
\]
The last inequality uses $r\ge(j+1)(j+2)\ge6$. For the comparison with $|xB_{r-1}(-x)|$, note that $\Lambda_j^r x^j\ge\Lambda_1^r x$ by \autoref{wh:weight-decay}. Since $\Lambda_1^{-1}=4\sin^2(\pi/10)<2/5$, by \eqref{wh:binomial-bound}, we have
\[
 \frac{|xB_{r-1}(-x)|}{(1+x)^{2r+1}\Lambda_j^r x^j}
 \le\frac{x^{1-j}}{(1+x)^{r+2}\Lambda_j^r}
 \le\Lambda_1^{-r}<\Bigl(\frac25\Bigr)^6<\frac1{50}.\qedhere
\]
\end{proof}

\subsection{The conjugate-pole estimate}\label{wh:central-range}

When $r<(j+1)(j+2)$, we take comparison points of the form $-e^{-s}$ with $s\ge0$.
At $s=0$ the series in \eqref{wh:alternating-ehrhart} no longer converges. We replace this series by a partial-fraction expansion that remains valid at $s=0$.

The reflection in \eqref{eq:count-reflection} is centered at $-3/2$, so the crown Ehrhart polynomial has only even powers in $q+3/2$:
\begin{equation}\label{wh:shifted-count}
 \ehr(\OO(\CC_{2r}),q)=\sum_{\ell=0}^r c_{r,\ell}(q+3/2)^{2r-2\ell}.
\end{equation}
By \eqref{eq:odd-sum} and \eqref{eq:crown-count}, $c_{r,\ell}>0$ for $0\le\ell<r$.

For $0<x<1$ and each $0\le\ell\le r$, the series in $q$ below converges absolutely by the ratio test. Substituting \eqref{wh:shifted-count} into \eqref{wh:alternating-ehrhart} and interchanging the sums, we obtain
\begin{equation}\label{wh:shifted-ehrhart-series}
 \frac{h_{r,0}(-x)}{(1+x)^{2r+1}}
 =\sum_{q\ge0}(-1)^q\ehr(\OO(\CC_{2r}),q)x^q
 =\sum_{\ell=0}^r c_{r,\ell}\sum_{q\ge0}(-1)^q(q+3/2)^{2r-2\ell}x^q.
\end{equation}
For $\ell<r$, the following lemma rewrites the inner series in \eqref{wh:shifted-ehrhart-series} by pairing the contributions from the conjugate poles of $1/(2\cosh(s/2))$. These poles are $s=\pm\mathrm{i}(2k+1)\pi$ for integers $k\ge0$, with $\pm\mathrm{i}\pi$ closest to the real axis.
The formula is a shifted form of the Lindel\"of-Wirtinger expansion \cite[Lemma~2.4]{GawronskiNeuschel}. We include a direct proof from the cosecant expansion.

\begin{lemma}\label{wh:shifted-power-sum}
Let $m\ge1$ be an integer. For $s>0$ and $x=e^{-s}$, we have
\begin{equation}\label{wh:real-pf}
 \sum_{q\ge0}(-1)^q(q+3/2)^m x^q
 =\frac{1}{2^m x}+\frac{2m!}{x^{3/2}}\operatorname{Im}
 \sum_{k\ge0}\frac{(-1)^k}{(s+\mathrm{i}(2k+1)\pi)^{m+1}}.
\end{equation}
\end{lemma}

\begin{proof}
The geometric series $1/(2\cosh(s/2))=\sum_{n\ge0}(-1)^n e^{-(n+1/2)s}$ admits termwise differentiation for $s>0$. Its derivative of order $m$, with the term $n=0$ separated and the remaining terms reindexed, is
\[
 \frac{d^m}{ds^m}\frac{1}{2\cosh(s/2)}
 =(-1)^m\Big[\frac{x^{1/2}}{2^m}
 -x^{3/2}\sum_{q\ge0}(-1)^q(q+3/2)^m x^q\Big].
\]
For another expression for the derivative, set $z=(\pi+\mathrm{i}s)/2$, so that $\tfrac12\csc z=1/(2\cosh(s/2))$. The cosecant expansion \cite[p.~118, equation~4.22.5]{NISTHandbook} and its derivatives take the form
\[
 \csc z=\frac1z+\sum_{n\ge1}(-1)^n\Big(\frac1{z-n\pi}+\frac1{z+n\pi}\Big),\qquad \frac{d^m}{dz^m}\csc z=(-1)^m m!\sum_{n\in\Z}\frac{(-1)^n}{(z-n\pi)^{m+1}}.
\]
Here $z\notin\pi\Z$ and $m\ge1$ is an integer. The derivative series converges absolutely and locally uniformly on $\mathbb C\setminus\pi\Z$, justifying termwise differentiation and regrouping. Since $2(z-n\pi)=\mathrm{i}[s+\mathrm{i}(2n-1)\pi]$, pairing $n=-k$ with $n=k+1$ gives
\[
 \frac{d^m}{ds^m}\frac{1}{2\cosh(s/2)}
 =2(-1)^{m+1}m!\operatorname{Im}
 \sum_{k\ge0}\frac{(-1)^k}{(s+\mathrm{i}(2k+1)\pi)^{m+1}}.
\]
Comparing the derivative formulas obtained from the geometric and cosecant series proves \eqref{wh:real-pf}.
\end{proof}

In \eqref{wh:shifted-ehrhart-series}, we apply \autoref{wh:shifted-power-sum} with $m=2r-2\ell$ for $\ell<r$. For $\ell=r$, the inner sum is $\sum_{q\ge0}(-1)^q x^q=1/(1+x)$. To combine these terms, we use the value at $q=-1$ in \eqref{eq:count-reflection}:
\[
 \sum_{\ell=0}^r 2^{2\ell-2r}c_{r,\ell}=\ehr(\OO(\CC_{2r}),-1)=0.
\]
Thus, the terms $c_{r,\ell}/(2^{2r-2\ell}x)$ with $\ell<r$ sum to $-c_{r,r}/x$. Combining this sum with $c_{r,r}/(1+x)$ and substituting $x=e^{-s}$, we obtain
\begin{equation}\label{wh:logistic}
 \begin{aligned}
 \frac{h_{r,0}(-e^{-s})}{(1+e^{-s})^{2r+1}e^{3s/2}}
 &=2\operatorname{Im}\sum_{k\ge0}(-1)^k\sum_{\ell=0}^{r-1}
 \frac{(2r-2\ell)!c_{r,\ell}}{(s+\mathrm{i}(2k+1)\pi)^{2r-2\ell+1}}-\frac{c_{r,r}}{2\cosh(s/2)}.
 \end{aligned}
\end{equation}
The pole series converges locally uniformly for real $s$, since its summands are $O(k^{-3})$, uniformly on compact real intervals. Therefore, both sides of \eqref{wh:logistic} are continuous at $s=0$, so the identity holds there as well.

To obtain the alternating signs in \eqref{wh:sign-comparison}, we choose $s$ using the contribution with $k=\ell=0$ in \eqref{wh:logistic}.
Write $s+\mathrm{i}\pi=\rho e^{\mathrm{i}\theta}$, where $\rho=\sqrt{s^2+\pi^2}$ and $0<\theta\le\pi/2$ for $s\ge0$.
The contribution with $k=\ell=0$ is a positive multiple of $-\rho^{-2r-1}\sin((2r+1)\theta)$. To make its sign $(-1)^j$, we choose
\begin{equation}\label{wh:nodes}
 \theta_j=\frac{\pi(j+3/2)}{2r+1},\qquad s_j=\pi\cot\theta_j,\qquad t_j=-e^{-s_j},\qquad \rho_j=\sqrt{s_j^2+\pi^2},
\end{equation}
for $r\ge4$ and $1\le j\le r-1$ with $r<(j+1)(j+2)$. Then $\theta=\theta_j$ and $\rho=\rho_j$ at $s=s_j$, and $-\sin((2r+1)\theta_j)=(-1)^j$. In particular, $s_{r-1}=0$. 

At each comparison point $s_j$, we divide \eqref{wh:logistic} by the modulus $2(2r)!c_{r,0}/\rho_j^{2r+1}$ of the term with $k=\ell=0$. The normalized contribution from the nearest pole pair ($k=0$) is
\begin{equation}\label{wh:nearest-pair}
 M_j:=(-1)^j-\sum_{\ell=1}^{r-1}
 \frac{(2r-2\ell)!c_{r,\ell}\rho_j^{2\ell}}{(2r)!c_{r,0}}
 \sin((2r-2\ell+1)\theta_j).
\end{equation}
To bound the remaining terms in \eqref{wh:logistic} and $t_jB_{r-1}(t_j)$, define
\begin{equation}\label{wh:error-definitions}
 \begin{aligned}
 \varepsilon_{r,1}(s_j)
 &:=\frac{\rho_j^{2r+1}}{(2r)!c_{r,0}}
 \sum_{k\ge1}\sum_{\ell=0}^{r-1}
 \frac{(2r-2\ell)!c_{r,\ell}}
 {[s_j^2+(2k+1)^2\pi^2]^{r-\ell+1/2}},\\
 \varepsilon_{r,2}(s_j)
 &:=\frac{|c_{r,r}|\rho_j^{2r+1}}
 {4(2r)!c_{r,0}\cosh(s_j/2)},\qquad
 \varepsilon_{r,3}(s_j)
 :=\frac{\rho_j^{2r+1}|e^{-s_j}B_{r-1}(-e^{-s_j})|}
 {2(2r)!c_{r,0}(1+e^{-s_j})^{2r+1}e^{3s_j/2}}.
 \end{aligned}
\end{equation}
After this normalization, the combined contribution of the pole pairs with $k\ge1$ has modulus at most $\varepsilon_{r,1}(s_j)$. The term $-c_{r,r}/(2\cosh(s_j/2))$ has normalized modulus $\varepsilon_{r,2}(s_j)$. The quantity $\varepsilon_{r,3}(s_j)$ is $|t_jB_{r-1}(t_j)|$ divided by both the denominator on the left-hand side of \eqref{wh:logistic} and the normalizing factor. Since both divisors are positive, the triangle inequality reduces \eqref{wh:sign-comparison} to
\begin{equation}\label{eq:sign-comparison}
 (-1)^jM_j-\varepsilon_{r,1}(s_j)-\varepsilon_{r,2}(s_j)-\varepsilon_{r,3}(s_j)>0.
\end{equation}
The next two lemmas establish $(-1)^jM_j\ge11/16$ and $\varepsilon_{r,1}(s_j)+\varepsilon_{r,2}(s_j)+\varepsilon_{r,3}(s_j)<2/3$.

\begin{lemma}\label{wh:nearest-coefficient-dominance}
Let $r\ge4$ and $1\le j\le r-1$ be integers with $r<(j+1)(j+2)$. Then 
\begin{equation}\label{eq:nearest-coefficient-dominance}
 \rho_j^2<5r,\qquad
 \frac{5r(2r-2)!c_{r,1}}{(2r)!c_{r,0}}\le\frac5{16},\qquad
 \sum_{\ell=2}^{r-1}
 \frac{(5r)^\ell(2r-2\ell)!c_{r,\ell}}{(2r)!c_{r,0}}
 \le\frac5{16}.
\end{equation}
Consequently, $(-1)^jM_j\ge11/16$. 
\end{lemma}

\begin{proof}
Since $r\le(j+1)(j+2)-1$, we have $r+1+1/(4r)<(j+3/2)^2$. The chord bound $\sin\theta\ge2\sqrt{2}\theta/\pi$ for sine on $[0,\pi/4]$ and its monotonicity on $[0,\pi/2]$ imply
\[
 r\sin^2\theta_j\ge
 \min\Bigl\{\frac r2,\frac{8r(j+3/2)^2}{(2r+1)^2}\Bigr\}\ge2.
\]
Hence, $\rho_j^2=\pi^2/\sin^2\theta_j\leq\pi^2r/2<5r$. Substitute the expression for $S_r$ in \eqref{eq:odd-sum} into \eqref{eq:crown-count}, with $a_{r,j}$ as in \eqref{eq:cosecant-coefficients}. Comparing coefficients of $(q+3/2)^{2r-2\ell}$ with \eqref{wh:shifted-count} and dividing by the leading coefficient $c_{r,0}$, we deduce
\begin{equation}\label{wh:coefficient-ratios}
 \frac{c_{r,\ell}}{c_{r,0}}
 =\Bigl(\frac{\pi^2}{4}\Bigr)^\ell
 \frac{(2r-2\ell-1)!}{(2r-1)!}e_{r,\ell}
 \frac{\sum_{n\ge0}(2n+1)^{-2r+2\ell}}
 {\sum_{n\ge0}(2n+1)^{-2r}},\qquad 0\le\ell<r.
\end{equation}
Since $(2n+1)^2>4n(n+1)$ and $(2n+1)^2\ge9$ for $n\ge1$, for every integer $m\ge1$ we have
\[
 \sum_{n\ge1}(2n+1)^{-2m}
 <\frac1{4\cdot9^{m-1}}\sum_{n\ge1}\frac1{n(n+1)}
 =\frac1{4\cdot9^{m-1}}.
\]
The denominator of the ratio of odd-integer sums in \eqref{wh:coefficient-ratios} is at least one. Hence, this ratio is less than $5/4$, and less than $37/36$ for $\ell=1,2$. In \eqref{wh:coefficient-ratios}, substitute $e_{r,1}=r(r-1)(2r-1)/6$ for $\ell=1$ and use the Cauchy-Schwarz bound $2e_{r,2}\le(r-2)e_{r,1}^2/(r-1)$ for $\ell=2$. Bounding the ratio of odd-integer sums by $37/36$ gives
\begin{equation}\label{wh:first-coefficient-bounds}
 \frac{5r(2r-2)!c_{r,1}}{(2r)!c_{r,0}}<\frac{37}{36}\frac{5\pi^2}{96}\frac{r}{2r-1}<\frac5{16},\quad
 \frac{(5r)^2(2r-4)!c_{r,2}}{(2r)!c_{r,0}}<\frac{37}{36}\frac{25\pi^4}{18432} \frac{r}{r-1}\Bigl(\frac{r}{2r-3}\Bigr)^2<\frac18.
\end{equation}
Here we used $\pi^2<10$, $r/(2r-1)\le4/7$, $r/(r-1)\le4/3$, and $r/(2r-3)\le4/5$.

For $3\le\ell<r$, each product in $e_{r,\ell}$ is bounded by the product of the largest $\ell$ squares, so
\[
 e_{r,\ell}\le\binom{r-1}{\ell}\prod_{k=r-\ell+1}^r(k-1)^2
 =\frac1{\ell!}\prod_{k=r-\ell+1}^r(k-1)^3.
\]
Substitute the bound on $e_{r,\ell}$ into
\eqref{wh:coefficient-ratios} and bound the ratio of odd-integer sums by $5/4$.
Simplifying the factorial ratios yields
\begin{equation}\label{wh:higher-coefficient-bound}
 \frac{(5r)^\ell(2r-2\ell)!c_{r,\ell}}{(2r)!c_{r,0}}
 <\frac5{4\ell!}\Bigl(\frac{5\pi^2}{64}\Bigr)^\ell
 \prod_{k=r-\ell+1}^r\frac{r}{k}\Bigl(\frac{2(k-1)}{2k-1}\Bigr)^2.
\end{equation}
For integers $k$ with $2\le k<r$, we have
\[
 \frac{r}{k}\Bigl(\frac{2(k-1)}{2k-1}\Bigr)^2
 \le\frac{r(k-1)}{k^2}
 \le\frac{(r-k+1)^2}{4(r-k)}.
\]
For the last inequality, the difference between the right- and left-hand sides is $[k(r-k+1)-2r]^2/[4k^2(r-k)]\ge0$. At $k=r$, the factor $(2r-2)^2/(2r-1)^2$ is at most $1$. Substituting these product bounds into \eqref{wh:higher-coefficient-bound}, we obtain
\begin{equation}\label{wh:higher-coefficient-bound-final}
 \frac{(5r)^\ell(2r-2\ell)!c_{r,\ell}}{(2r)!c_{r,0}}
 \mathrel{<}\frac5{4\ell!}\Bigl(\frac{5\pi^2}{64}\Bigr)^\ell
 \underbrace{\prod_{k=r-\ell+1}^{r-1}\frac{(r-k+1)^2}{4(r-k)}}_{\ell\,\ell!/4^{\ell-1}}
 \mathrel{=}5\ell\Bigl(\frac{5\pi^2}{256}\Bigr)^\ell
 <\frac{\ell}{5^{\ell-1}}.
\end{equation}
Adding these bounds to the $\ell=2$ estimate in \eqref{wh:first-coefficient-bounds}, we derive
\[
 \sum_{\ell=2}^{r-1}
 \frac{(5r)^\ell(2r-2\ell)!c_{r,\ell}}{(2r)!c_{r,0}}
 <\frac18+\sum_{\ell\ge3}\frac{\ell}{5^{\ell-1}}
 =\frac{23}{80}<\frac5{16}.
\]

To estimate $M_j$ in \eqref{wh:nearest-pair}, first rewrite the sine factors using \eqref{wh:nodes}:
\[
 \sin((2r-2\ell+1)\theta_j)
 =(-1)^{j+1}\cos(2\ell\theta_j).
\]
For $\ell=1$, use $\cos(2\theta_j)=1-2\sin^2\theta_j$. For $\ell\ge2$, bound $\cos(2\ell\theta_j)$ below by $-1$. With $\rho_j\sin\theta_j=\pi$, $2\pi^2<5r$, and $\rho_j^2<5r$, the resulting bounds are
\[
 \rho_j^2\cos(2\theta_j)
 =\rho_j^2-2\pi^2\ge\rho_j^2-5r,\quad
 \rho_j^{2\ell}\cos(2\ell\theta_j)
 \ge-\rho_j^2(5r)^{\ell-1},\qquad 2\le\ell<r.
\]
In $(-1)^jM_j$, the sine identity above turns the terms with $\ell\ge1$ into cosines with positive coefficients. The trigonometric lower bounds apply term by term, and \eqref{eq:nearest-coefficient-dominance} bounds the resulting losses for $\ell=1$ and $\ell\ge2$:
\[
 (-1)^jM_j
 \ge1-\frac{5r-\rho_j^2}{16r}
 -\frac{\rho_j^2}{5r}\cdot\frac5{16}
 =\frac{11}{16}.\qedhere
\]
\end{proof}

We next estimate the three quantities defined in \eqref{wh:error-definitions}.

\begin{lemma}\label{wh:direct-error-bounds}
Let $r\ge4$ and $1\le j\le r-1$ be integers with $r<(j+1)(j+2)$. Then
\[
 \varepsilon_{r,1}(s_j)<\frac16,\qquad
 \varepsilon_{r,2}(s_j)<\frac16,\qquad
 \varepsilon_{r,3}(s_j)<\frac13.
\]
\end{lemma}

\begin{proof}
In the definition of $\varepsilon_{r,1}(s_j)$ in \eqref{wh:error-definitions}, we substitute $s_j^2+(2k+1)^2\pi^2=\rho_j^2+4\pi^2k(k+1)$. The bound $\rho_j^2<5r$ from \eqref{eq:nearest-coefficient-dominance} then implies
\begin{equation}\label{wh:remaining-pole-sum}
 \varepsilon_{r,1}(s_j)\mathrel{<}\sum_{\ell=0}^{r-1}
 \frac{(2r-2\ell)!c_{r,\ell}(5r)^\ell}{(2r)!c_{r,0}}
 \sum_{k\ge1}\Big(1+\frac{4\pi^2k(k+1)}{5r}\Big)^{-(r-\ell+1/2)}.
\end{equation}
For $0\le\ell<r$, we have $(2r-2\ell+1)(\ell+3/2)>3r$ and $(2r-2\ell+1)/3\ge1$. Bernoulli's inequality $(1+y)^a\ge1+ay$ for $y\ge0$ and $a\ge1$ implies
\[
 \Big(1+\frac{4\pi^2k(k+1)}{5r}\Big)^{-(r-\ell+1/2)}\le\Big(1+\frac{4\pi^2k(k+1)(2r-2\ell+1)}{15r}\Big)^{-3/2}<\Big(\frac{5(\ell+3/2)}{4\pi^2k(k+1)}\Big)^{3/2}.
\]
To bound the inner sum in \eqref{wh:remaining-pole-sum}, we first apply the arithmetic-geometric mean inequality:
\[
 \sum_{k\ge1}\frac{\sqrt{k(k+1)}}{k^2(k+1)^2}<\sum_{k\ge1}\frac{2k+1}{2k^2(k+1)^2}\mathrel{=}\frac12\sum_{k\ge1}\Big(\frac1{k^2}-\frac1{(k+1)^2}\Big)=\frac12.
\]
Combining the two estimates bounds the inner sum in \eqref{wh:remaining-pole-sum} by $\frac{5\sqrt5}{16\pi^3}(\ell+\tfrac32)^{3/2}$. The bounds \eqref{wh:first-coefficient-bounds} and \eqref{wh:higher-coefficient-bound-final}, together with the term $\ell=0$, imply
\[
 \begin{aligned}
 &\sum_{\ell=0}^{r-1}
 \frac{(2r-2\ell)!c_{r,\ell}(5r)^\ell}{(2r)!c_{r,0}}
 (\ell+\tfrac32)^{3/2}
 <\Big(\frac32\Big)^{3/2}
 +\frac5{16}\Big(\frac52\Big)^{3/2}
 +\frac18\Big(\frac72\Big)^{3/2}
 +\frac{5/4}{1-3/8}<6.
 \end{aligned}
\]
The first three bounds correspond to $\ell=0,1,2$. For $\ell\ge3$, the bounds $\ell(\ell+3/2)^{3/2}/5^{\ell-1}$ have first term less than $5/4$ and successive ratios at most $(4/15)(11/9)^{3/2}<3/8$. Consequently,

\[
 \varepsilon_{r,1}(s_j)<\frac{5\sqrt5}{16\pi^3}\cdot6
 =\frac{15\sqrt5}{8\pi^3}<\frac16,
\]
where the last inequality uses $\pi>3$ and $\sqrt5<12/5$.

The estimates for $\varepsilon_{r,2}(s_j)$ and $\varepsilon_{r,3}(s_j)$
use a common lower bound on $(2r)!c_{r,0}$.
Substitute \eqref{eq:odd-sum} into \eqref{eq:crown-count} and compare
leading coefficients to find $c_{r,0}$.
For $c_{r,r}$, evaluate the resulting polynomial at $q=-3/2$:
\[
 c_{r,0}=\Bigl(\frac2\pi\Bigr)^{2r}
 \sum_{n\ge0}(2n+1)^{-2r}
 \ge\Bigl(\frac2\pi\Bigr)^{2r},
 \qquad |c_{r,r}|=\frac1{2\cdot4^r}.
\]
For $r\ge5$, pair the factors of $(2r)!/8!$ from opposite ends. Each pair satisfies $k(2r+9-k)=18r+(k-9)(2r-k)\ge18r$ for $9\le k\le r+4$. The resulting bound $(2r)!/8!\ge(18r)^{r-4}$ also holds at $r=4$. Hence,
\begin{equation}\label{wh:factorial-lower-bound}
 (2r)!c_{r,0}
 \ge8!\Bigl(\frac2\pi\Bigr)^8
 \Bigl(\frac{72r}{\pi^2}\Bigr)^{r-4}
 >1080(5r)^{r-4},
\end{equation}
since $8!(2/\pi)^8>8!(7/11)^8>1080$ and $\pi^2<10$. We also use
\begin{equation}\label{wh:cosh-lower-bound}
 \cosh(s_{j}/2)\ge1+s_{j}^2/8\ge\rho_{j}^2/\pi^2,
\end{equation}
where the last inequality follows from $\pi^2>8$.

For $\varepsilon_{r,2}(s_j)$, substitute $|c_{r,r}|=1/(2\cdot4^r)$ into its definition \eqref{wh:error-definitions}. Since $\rho_j^2<5r$, the lower bounds \eqref{wh:factorial-lower-bound} and \eqref{wh:cosh-lower-bound} imply
\[
 \varepsilon_{r,2}(s_j)
 <\frac{\rho_j^9}{8\cdot4^r\cdot1080\cosh(s_j/2)}
 \le\frac{\pi^2(5r)^{7/2}}{8\cdot4^r\cdot1080}
 \le\frac{25\pi^2\sqrt5}{3456}<\frac16.
\]
Here $(5r)^{7/2}/4^r$ decreases for $r\ge4$, because its successive ratio is $(1+1/r)^{7/2}/4<1$. The last inequality uses $\pi^2<10$ and $\sqrt5<9/4$.

Finally, applying \eqref{wh:binomial-bound} and \eqref{wh:factorial-lower-bound} to $\varepsilon_{r,3}(s_j)$ in \eqref{wh:error-definitions}, we find
\[
 \varepsilon_{r,3}(s_j)\mathrel{<}\frac{\rho_{j}^{2r+1}e^{-5s_{j}/2}}
 {2\mathbin{\cdot}1080(5r)^{r-4}(1+e^{-s_{j}})^{r+2}}
 \mathrel{\le}\frac{\rho_{j}^9e^{-5s_{j}/2}}{2\cdot1080(1+e^{-s_{j}})^6},
\]
where the second inequality uses $r\ge4$, $\rho_j^2<5r$, and $1+e^{-s_j}>1$.
Rewriting the last bound in terms of $\cosh(s_j/2)$, we obtain
\[
 \varepsilon_{r,3}(s_j)
 <\frac{\rho_{j}^9e^{s_{j}/2}}{128\cdot1080\cosh^6(s_{j}/2)}
 \le\frac{\pi^9}{128\cdot1080}
 \frac{e^{s_{j}/2}}{\cosh^{3/2}(s_{j}/2)}.
\]
The last inequality uses \eqref{wh:cosh-lower-bound}. The logarithmic derivative of $e^{s/2}/\cosh^{3/2}(s/2)$ is $1/2-(3/4)\tanh(s/2)$. It changes sign from positive to negative at $s=\log5$, so the function attains its maximum there on $s\ge0$. Hence,
\[
 \frac{e^{s_{j}/2}}{\cosh^{3/2}(s_{j}/2)}
 \le\frac{5^{5/4}}{3^{3/2}}<\frac32.
\]
Thus, $\varepsilon_{r,3}(s_j)<3\pi^9/(256\cdot1080)<1/3$, where the last inequality uses $\pi<22/7$.
\end{proof}

Since $11/16-2/3=1/48>0$, the bounds in \autoref{wh:nearest-coefficient-dominance} and \autoref{wh:direct-error-bounds} verify \eqref{eq:sign-comparison}, proving the following lemma.

\begin{lemma}\label{wh:central}
Let $r\ge4$ and $1\le j\le r-1$ be integers with $r<(j+1)(j+2)$. At the point $t_j$ in \eqref{wh:nodes},
\[
 (-1)^j h_{r,0}(t_j)>|t_j B_{r-1}(t_j)|.
\]
\end{lemma}

\subsection{Proof of \autoref{thm:wheel-whirl}}\label{wh:completion}

For $r\ge4$, \autoref{wh:interior-endpoint} and \autoref{wh:central} establish \eqref{wh:sign-comparison} in the ranges $r\ge(j+1)(j+2)$ and $r<(j+1)(j+2)$, respectively. It remains to verify that $0>t_1>\cdots>t_{r-1}=-1$.

\begin{proof}[Proof of \autoref{thm:wheel-whirl}]
The case $r=3$ was settled in \autoref{wh:signs}. 
For integers $r\ge4$ and $1\le j\le r-1$, the comparison points from \eqref{wh:endpoint-node} and \eqref{wh:nodes} can be written as
\[
 s_j=\begin{cases}
 r\log\dfrac{2j+5}{2j+1},&r\ge(j+1)(j+2),\\[4pt]
 \pi\cot\dfrac{\pi(j+3/2)}{2r+1},&r<(j+1)(j+2),
 \end{cases}
 \qquad t_j=-e^{-s_j}.
\]
The first range is a possibly empty initial segment, and $s_j$ decreases strictly on each range.
 At the last index $j=r-1$, the inequality $r<r(r+1)$ selects the second range, so $s_{r-1}=0$.

If both ranges occur, it remains to check that $s_j$ decreases strictly across their transition. Suppose the first range ends at $j$, and set $u:=j+3/2$. Since $r\ge(j+1)(j+2)$, we have $2r>u$ and $0<\pi(u+1)/(2r+1)<\pi/2$. Using $\log((u+1)/(u-1))>2/u$ for $u>1$ and $\cot\theta<1/\theta$ for $0<\theta\le\pi/2$, we obtain
\[
 s_j=-\log x_j
 =r\log\frac{u+1}{u-1}
 >\frac{2r}{u}
 >\frac{2r+1}{u+1}
 >\pi\cot\frac{\pi(u+1)}{2r+1}
 =s_{j+1}.
\]
Since $s\mapsto-e^{-s}$ is strictly increasing, the ordering of the $s_j$ implies $0>t_1>\cdots>t_{r-1}=-1$. The bounds from \autoref{wh:interior-endpoint} and \autoref{wh:central} supply the sign comparison required by \autoref{lem:sign-comparison}. Therefore, $h_{r,\lambda}$ has $2r-2$ simple negative roots for every $\lambda\in[0,1]$. The endpoints $\lambda=0$ and $\lambda=1$ correspond to $h^*_{P(N_r)}$ and $h^*_{P(M_r)}$, respectively. The crown case follows from \autoref{thm:crown-whirl}.
\end{proof}

\section{Applications}\label{sec:closure}

A \emph{stable set} in a finite simple graph $G$ is a set of pairwise nonadjacent vertices. The \emph{stable-set polytope}
$\operatorname{STAB}(G)$ is the convex hull of the indicator vectors
of the stable sets of $G$.

For each integer $r\ge2$, write $C_{2r}$ for the cycle on $1,\ldots,2r$ in cyclic order.
Jiang, Yang, and Zhong \cite[Theorem~1.5 with $a=1$ and $m=2r$]{JiangYangZhong} established symmetry and unimodality of $h^*_{\operatorname{STAB}(C_{2r})}$.
The edges of $C_{2r}$ are precisely the comparable pairs of $\CC_{2r}$. We use this correspondence to transfer real-rootedness from the crown order polytope.

\begin{corollary}\label{cor:even-cycle}
For every integer $r\ge3$, the polynomial
$h^*_{\operatorname{STAB}(C_{2r})}(t)$ has $2r-2$ distinct negative roots.
For $r=2$, it equals $(1+t)^2$.
\end{corollary}

\begin{proof}
Complement the even coordinates, as in Stanley's affine equivalence between the order and chain polytopes \cite[Theorem~2.3]{StanleyPosets}:
\[
 \boldsymbol{x}\mapsto\boldsymbol{y}=(x_1,1-x_2,\ldots,x_{2r-1},1-x_{2r}).
\]
For each edge $ij$ with $i$ odd and $j$ even, the crown inequality $x_i\le x_j$ becomes $y_i+y_j\le1$. Thus, the map takes the crown's $0/1$ points onto the stable-set indicators. Taking convex hulls and using Stanley's vertex description \cite[Corollary~1.3]{StanleyPosets} identifies the image with $\operatorname{STAB}(C_{2r})$. Moreover, the linear part is diagonal with entries $\pm1$, and the translation is integral. Hence, the map is a unimodular equivalence and
\[
 h^*_{\operatorname{STAB}(C_{2r})}(t)=h^*_{\OO(\CC_{2r})}(t).
\]
The root count for $r\ge3$ and the formula for $r=2$ follow from \autoref{thm:wheel-whirl} and \eqref{eq:crown-rank-two}, respectively.
\end{proof}

The real-rootedness results in \autoref{thm:wheel-whirl}, \eqref{eq:crown-rank-two}, and \autoref{cor:even-cycle} extend to products and positive integer dilations.

\begin{corollary}\label{cor:products}
Let $P=P_1\times\cdots\times P_s$, where $s\ge1$ is an integer. For each $j$, the factor $P_j$ is either $\OO(\CC_{2r})$ or $\operatorname{STAB}(C_{2r})$ for some integer $r\ge2$, or $P(M_r)$ or $P(N_r)$ for some integer $r\ge3$. For every integer $k\ge1$, the polynomial $h^*_{kP}$ has only negative real roots, and its coefficient sequence is ultra log-concave of order $\deg h^*_{kP}$.
\end{corollary}

\begin{proof}
Each $h_{P_j}^*$ has only negative real roots by
\autoref{thm:wheel-whirl}, \eqref{eq:crown-rank-two}, and \autoref{cor:even-cycle}.
The equality $qP=(qP_1)\times\cdots\times(qP_s)$ implies $\ehr(P,q)=\prod_{j=1}^s\ehr(P_j,q)$ for $q\in\Z_{\ge0}$. Thus, the Ehrhart series of $P$ is the coefficientwise product of the factor series. Since each Ehrhart polynomial has degree equal to the dimension of its polytope, the numerators in Wagner's theorem \cite[Theorem~0.2]{Wagner} are the $h^*$-polynomials in \eqref{eq:ehrhart-series}. Successive applications of Wagner's theorem show that $h_P^*$ has only nonpositive real roots. Its constant term $1$ excludes zero.

Since $q(kP)=(kq)P$ and $\dim(kP)=\dim P$, the Ehrhart series of $kP$ is
\[
 \frac{h_{kP}^*(t)}{(1-t)^{\dim P+1}}=\sum_{q\ge0}\ehr(P,kq)t^q.
\]
The coefficients on the right are those at indices divisible by $k$ in the Ehrhart series of $P$. Since $h_P^*$ has nonnegative coefficients and only negative real roots, Zhang's result \cite[Corollary~3.4]{Zhang} implies that $h_{kP}^*$ is real-rooted. Its nonnegative coefficients and constant term $1$ exclude roots on $[0,\infty)$. Factoring $h_{kP}^*$ over its negative roots shows that its coefficients of degrees $0,\ldots,\deg h_{kP}^*$ are positive. By Newton's inequalities \cite[Theorem~2]{StanleyLogConcavity}, the coefficient sequence is ultra log-concave of order $\deg h_{kP}^*$.
\end{proof}

For matroids $M,N$ on disjoint ground sets, a basis of $M\oplus N$ is the union of a basis of $M$ and a basis of $N$. Taking convex hulls of the corresponding pairs of indicator vectors gives
\[
 P(M\oplus N)=P(M)\times P(N).
\]
Thus, \autoref{cor:products} applies to base polytopes of finite direct sums
of wheel and whirl matroids of rank at least three and to their positive integer dilates.

For a graph formed from wheels by disjoint unions or by identifying single vertices of disjoint graphs, every cycle lies in one wheel. Thus, the cycle matroid of the graph is a direct sum of wheel matroids. For every integer $q\ge0$, a bridge or a loop fixes the coordinate indexed by that edge at $q$ or $0$, respectively, in the $q$-dilate of the base polytope. Deleting that coordinate preserves dimension and induces a lattice-point bijection. Hence, adding either edge leaves the $h^*$-polynomial unchanged.

\section*{Acknowledgements}
This work is supported by the Guangdong Basic and Applied Basic Research
Foundation (Grant No. 2026A1515012237).

\section*{Declaration of AI use}
The authors developed the original results and main methods. They used OpenAI's Codex to assist with the estimates and error bounds in Section~4 and with literature searches, proofreading, and language editing. The authors reviewed and verified the manuscript and take full responsibility for its content.

\end{document}